\documentclass[11pt]{article}

\usepackage{graphicx,amssymb,amsmath}
\usepackage{epsfig}
\usepackage{epsf}
\usepackage{amsfonts}
\usepackage{latexsym}
\usepackage{fancyvrb}
\usepackage[small]{caption}

\usepackage{amssymb}
\usepackage{amsthm}
\usepackage{graphics}
\usepackage{psfrag}

\newtheorem{theorem}{Theorem}
\newtheorem{lemma}{Lemma}
\newtheorem{corollary}{Corollary}

\newtheorem{conjecture}{Conjecture}

\newcommand{\etal}{{et~al.}}
\newcommand{\ie}{{i.e.}}

\newcommand{\old}[1]{{}}

\newcommand{\conv}{{\rm conv}}
\newcommand{\per}{{\rm per}}
\newcommand{\len}{{\rm len}}

\newcommand{\RR}{\mathbb{R}}

\def\G{\mathcal G}

\def\eps{\varepsilon}

\providecommand{\intd}[0]%
{\;\mbox{d}}

\title{\Large THE OPAQUE SQUARE}

\author{Adrian Dumitrescu\thanks{Department of Computer Science,
University of Wisconsin--Milwaukee, USA\@.
Email:~\texttt{dumitres@uwm.edu}.
Supported in part by NSF grant DMS-1001667.}
\and
Minghui Jiang\footnote{%
Department of Computer Science,
Utah State University,
Logan, USA\@.
Email: \texttt{mjiang@cc.usu.edu}.}
}

\begin{document}

\maketitle

\begin{abstract}
The problem of finding small sets that block every line passing
through a unit square was first considered by Mazurkiewicz in 1916. 
We call such a set {\em opaque} or a {\em barrier} for the square. 
The shortest known barrier has length $\sqrt{2}+ \frac{\sqrt{6}}{2}=
2.6389\ldots$. The current best lower bound for the length of a 
(not necessarily connected) barrier is $2$, as established by Jones about
50~years ago.   
No better lower bound is known even if the barrier is restricted to 
lie in the square or in its close vicinity. 
Under a suitable locality assumption, we replace this lower bound by
$2+10^{-12}$, which represents the first, albeit small, step in a long
time toward finding the length of the shortest 
barrier. A sharper bound is obtained for interior barriers: 
the length of any interior barrier for the unit square is at least
$2 + 10^{-5}$.
Two of the key elements in our proofs are: 
(i) formulas established by Sylvester for the measure of all lines that meet
two disjoint planar convex bodies, and
(ii) a procedure for detecting lines
that are witness to the invalidity of a short bogus barrier for the square. 


\end{abstract}


\section{Introduction} \label{sec:intro}

The problem of finding small sets that block every line passing
through a unit square was first considered by Mazurkiewicz in
1916~\cite{Ma16}; see also~~\cite{Ba59},~\cite{GM55}. 
Let $C$ be a convex body in the plane. Following Bagemihl~\cite{Ba59},
a set $\Gamma$ is an {\em opaque set} or a {\em barrier} for $C$, if $\Gamma$ 
meets all lines that intersect $C$. A barrier does not need to be
connected; it may consist of one or more rectifiable arcs and its 
parts may lie anywhere in the plane, including the exterior of $C$;
see~\cite{Ba59},~\cite{Br92}. 

{\em What is the length of the shortest barrier for a given convex
body $C$?} In spite of considerable efforts, the answer to this
question is not known even in the simplest instances,
such as when $C$ is a square, a disk, or an equilateral triangle;
see~\cite{Cr69},~\cite[Problem~A30]{CFG91},~\cite{E82},~\cite{FM86},~\cite{FMP84}, 
~\cite[Section~8.11]{F03},~\cite[Problem~12]{H78}.
Some entertaining variants of the problem appeared in different
forms in the literature~\cite{AG08,Br92,G90,J80,K86,K87}. 

A barrier blocks any line of sight across the region $C$ or detects
any ray that passes through it. Potential applications are in
guarding and surveillance~\cite{DO08}. Here we focus on the case
when $C$ is a square. The shortest barrier known for the unit square,
of length $2.639\ldots$, is illustrated in Figure~\ref{f1}~(right).
It is conjectured to be optimal. The current best lower bound, $2$,
has been established by Jones~\cite{J64} in 1964\footnote{A note of
  caution for the non-expert about the subtlety of the problem: 
  an arxiv submission~\cite{DP10} dated May 2010  
  claimed a first small improvement in the old lower bound of
  $2$ for a unit square, due to Jones~\cite{J64}, from 1964; its
  proof had a fatal error, and the submission was soon after withdrawn
  by the authors. Further, at least two 
  conference submissions by two other groups of authors were made
  in the last $3$ years claiming (erroneous) improvements in the same
  lower bound of $2$ for a unit square; both submissions were rejected
  at the respective conferences and the authors were notified
  of the errors discovered. In September 2013, yet another
  improvement in the lower bound of $2$ for a unit square has been
  announced~\cite{KMO13}. Its correctness however 
  remains unverified, since no proof seems to be publicly available at
  the time of this writing.}.
\begin{figure}[htbp]
\centerline{\epsfxsize=6.3in \epsffile{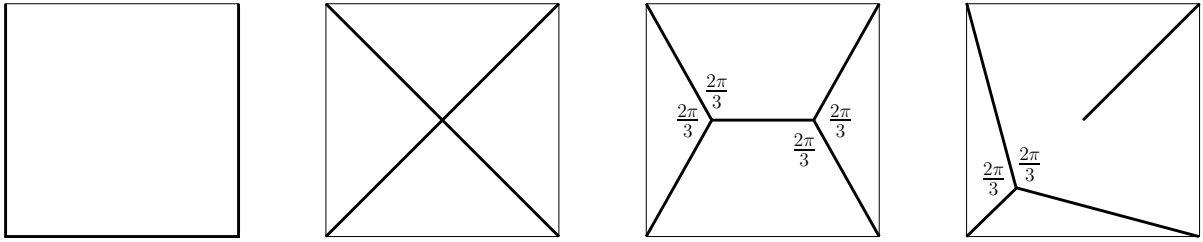}}
\caption{The first three from the left are barriers for the unit
square of lengths $3$, $2\sqrt{2}=2.8284\ldots$, and $1 +\sqrt{3}= 2.7320\ldots$.
Right: The diagonal segment $[(1/2,1/2),(1,1)]$ together with
  three segments connecting the corners $(0,1)$, $(0,0)$, $(1,0)$ to
  the point $(\frac{1}{2}-\frac{\sqrt3}{6},\frac{1}{2}-\frac{\sqrt3}{6})$
  yield a barrier of length $\sqrt2 + \frac{\sqrt6}{2}= 2.639\ldots$. }
\label{f1}
\end{figure}

The type of curve barriers considered may vary: the most restricted are
barriers made from single continuous arcs, then connected barriers,
and lastly, arbitrary (possibly disconnected) barriers. For the unit
square, the shortest known in these three categories have lengths $3$, 
$1 +\sqrt{3}= 2.7320\ldots$ and $\sqrt2 + \frac{\sqrt6}{2}=
2.6389\ldots$, respectively. They are depicted in Figure~\ref{f1}. 
Obviously, disconnected barriers offer the greatest freedom of design. 
For instance, Kawohl~\cite{K00} showed that the barrier in
Figure~\ref{f1}~(right) is optimal in the class of curves with at most
two components restricted to the square. For the unit disk, the
shortest known barrier consists of three arcs. See also~\cite{FM86,F03}.

Barriers can be also classified by where they can be located. In
certain instances, it might be infeasible to construct barriers guarding
a specific domain outside the domain, since that part might belong to
others. Following~\cite{DJP12} we call such barriers
constrained to the interior and the boundary of the domain,
\emph{interior}. For example, all four barriers for the unit square
illustrated in Figure~\ref{f1} are interior barriers. 
On the other hand, certain instances may prohibit barriers lying in
the interior of a domain.  We call a barrier constrained to the
exterior and the boundary of the domain, \emph{exterior}. 
For example, since the first barrier from the left in Figure~\ref{f1} is
contained in the boundary of the square, it is also an exterior barrier. 

\paragraph{Early algorithms and other related work.}
Two algorithms, proposed by Akman~\cite{A87} and respectively
Dublish~\cite{Du88} in the late 1980s, claiming to compute shortest
interior-restricted barriers, were refuted by Shermer~\cite{Sh91}
in the early 1990s. Shermer~\cite{Sh91} proposed a new algorithm instead,
which shared the same fate and was refuted recently by 
Provan~\etal~\cite{PBTW12}. 
As of today, no exact algorithm for computing a shortest 
(interior-restricted or unrestricted) barrier is known. 
Even though we have so little control on the shape or length of optimal
barriers, barriers whose lengths are somewhat longer can be computed
efficiently for any given convex polygon.  
Various approximation algorithms with a small constant ratio have been
obtained recently by Dumitrescu~\etal~\cite{DJP12}.   

If instead of curve barriers, we want to find {\em discrete} barriers
consisting of as few points as possible with 
the property that every line intersecting $C$ gets closer than
$\eps>0$ to at least one of them in some fixed norm, we arrive
at a problem raised by L\'aszl\'o Fejes T\'oth~\cite{FT73,FT74}
and subsequently studied by others~\cite{BF85,KW90,MP83,RS03,V94}. 
The problem of short barriers has attracted many other researchers and has
been studied at length; see also~\cite{Cr69,EP80,H78,M80}.

\paragraph{Our Results.}

In Section~\ref{sec:T1}, we prove:
\begin{theorem} \label{T1}
The length of any barrier for the unit square $U$
restricted to the square of side length $2$ concentric and homothetic
to $U$ is at least $2+10^{-12}$.
\end{theorem}

The possibility that parts of the barrier may be located outside of
the unit square $U$ only adds to the difficulty of obtaining a good
lower bound. Indeed, for the special case of barriers whose location
is restricted to $U$, the proof of the inequality in Theorem~\ref{T1} becomes
slightly easier. Moreover, a better lower bound can be obtained 
(along the same lines) in this case; however we omit this exercise. 
We then go one step further,
and by combining the methods developed in proving Theorem~\ref{T1}  
with the use of linear programming, in Section~\ref{sec:T2} we
establish a sharper bound: 
\begin{theorem} \label{T2}
The length of any interior barrier for the unit square is at least
$2 + 10^{-5}$.
\end{theorem}

\section{Preliminaries} \label{sec:prelim}

\paragraph{Definitions and notations.}

For a curve $\gamma$, let $|\gamma|$ denote the length of
$\gamma$. Similarly, if $\Gamma$ is a set of curves, let $|\Gamma|$ denote
the total length of the curves in $\Gamma$.
In order to be able to refer to the \emph{length} $\len(\Gamma)$ of a
barrier $\Gamma$,  we restrict our attention to rectifiable barriers. 
A \emph{rectifiable curve} is a curve of finite length. 
A \emph{rectifiable barrier} is the union of a countable set of
\emph{rectifiable curves}, $\Gamma=\cup_{i=1}^\infty \gamma_i$, 
where $\sum_{i=1}^\infty |\gamma_i| <\infty$ 
(or $\Gamma=\cup_{i=1}^n \gamma_i$ for some $n$).
A~\emph{segment barrier} is a barrier consisting of straight-line
segments (or polygonal paths). 
The shortest segment barrier is not much longer than the shortest
rectifiable one:
\begin{lemma} \label{L1} {\rm \cite{DJP12}}.
Let $\Gamma$ be a rectifiable barrier for a convex body $C$ in the plane.
Then, for any $\eps>0$, there exists a segment barrier $\Gamma_{\eps}$ for $C$,
consisting of a countable set of straight-line segments, such that
$\len(\Gamma_{\eps}) \leq (1+\eps)\, \len(\Gamma)$.
\end{lemma}

Without loss of generality, we will subsequently consider only segment barriers.
We first review three different proofs for the lower bound of $2$ 
(the current best lower bound for the unit square). 

\medskip
{\em First proof}:
The first proof, Lemma~\ref{L2}, is general and applies to any planar
convex body; its proof is folklore; see also~\cite{DJP12}. Let 
$\Gamma=\{s_1,\ldots,s_n\}$ consist of $n$ segments of lengths $\ell_i=|s_i|$, 
where $L=|\Gamma|=\sum_{i=1}^n \ell_i$.
Let $\alpha_i \in [0,\pi)$ be the angle made by $s_i$ with the $x$-axis. 
For each direction $\alpha \in [0,\pi)$, the blocking (opaqueness)
  condition for a convex body $C$ requires 
\begin{equation} \label{E1}
\sum_{i=1}^n \ell_i |\cos(\alpha-\alpha_i)| \geq w(\alpha).
\end{equation}
Here $w(\alpha)$ is the width of $C$ in direction $\alpha$, \ie, the
minimum width of a strip of parallel lines enclosing $C$, whose lines
are orthogonal to direction $\alpha$. Integrating this inequality
over the interval $[0,\pi]$ yields the following.
\begin{lemma}\label{L2}
Let $C$ be a convex body in the plane and let $\Gamma$ be a barrier for $C$.
Then the length of $\Gamma$ is at least $\frac12 \cdot \per(C)$.
\end{lemma}
For the unit square we have $\per(U)=4$ thus Lemma~\ref{L2} yields
the lower bound $L \geq 2$.

\medskip
{\em Second proof}:
We make use of formulas established by Sylvester~\cite{Sy1890}; see
also~\cite[pp.~32--34]{Sa04}. 
The setup is as follows. For a planar convex body $K$, the measure of
all lines that meet  $K$ is equal to $\per(K)$. In particular, if $K$
degenerates to a segment $s$, the measure of all lines that meet  $s$
is equal to $2|s|$.  

Let $\G$ denote all lines that meet $U$; 
let $\G_i$ denote all lines that meet a segment $s_i \in \Gamma$. 
The measure of all lines that meet $U$ is equal to $m(\G)=\per(U)=4$. 
Since $m()$ is a measure, we have
$$ 4 = m(\G) \leq \sum_{i=1}^n m(\G_i) = 2 \sum_{s_i \in \Gamma} |s_i|
=2L.$$
It follows that $2L \geq 4$ or $L \geq 2$, as required.

\medskip
{\em Third proof} (due to O.~Ozkan; reported in~\cite{DO08}):
This proof is specific to the square.
The setup is the same, in the sense that both  proofs assume,
without loss of generality by Lemma~\ref{L1}, a segment barrier. Let 
$\Gamma=\{s_1,\ldots,s_n\}$ consist of $n$ segments of lengths $\ell_i=|s_i|$, 
where $L=|\Gamma|=\sum_{i=1}^n \ell_i$.
Recall that $d_1$ and $d_2$ are the two diagonals of $U$.
Let $\theta_i \in [0,\pi)$ be the angle made by $s_i$ with
the first diagonal $d_1$. Consider
the blocking (opaqueness) conditions only for these two directions,
that is, for $\alpha=\pi/4$, and $\alpha=3\pi/4$.
Equation~\eqref{E1} for these two directions gives now:
\begin{equation} \label{E2}
\sum_{i=1}^n \ell_i |\cos \theta_i | \geq \sqrt{2}, \ \ {\rm and} \ \ 
\sum_{i=1}^n \ell_i |\sin \theta_i | \geq \sqrt{2}.
\end{equation}
Consequently, since $|\cos \theta_i |+ |\sin \theta_i | \leq
\sqrt{2}$ holds for any angle $\theta_i$, we have
\begin{equation} \label{E3}
2\sqrt{2} \leq \sum_{i=1}^n \ell_i (|\cos \theta_i|+
|\sin \theta_i |) \leq \sqrt2 \sum_{i=1}^n \ell_i
\quad\Rightarrow\quad \ L=\sum_{i=1}^n \ell_i \geq 2.
\end{equation}

An obvious question is whether any of these proofs can give more.
Regarding the third proof, if one considers only those four main
directions used there, namely the two coordinate axes and the two diagonal
directions, there is no hope left. 
Interestingly enough, there exists a structure (imperfect barrier) of
length $2$ made of four axis-parallel segments, that perfectly blocks
(\ie, with no overlap) the four main directions; see Figure~\ref{f2}.
Thus one needs to find other directions that are not opaque besides
these four. This observation was the starting point of our investigations.  
\begin{figure}[htbp]
\centerline{\epsfxsize=1.3in \epsffile{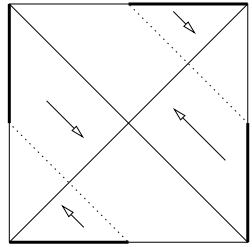}}
\caption{This structure of length $2$ perfectly blocks the four main
  directions; here shown for the main diagonal.}
\label{f2}
\end{figure}

\paragraph{Setup for the new lower bound of $2+10^{-12}$.}
We set four parameters:
\begin{itemize} \itemsep 1pt
\item  $\delta=10^{-12}$, \quad $\phi=\arcsin{10^{-4}}$; 
note that $10^8 \delta= \sin \phi$. 
\item $w_1=1/20$, \quad $w_2=1/1000$. 
\end{itemize}

Refer now to Figure~\ref{f3} which illustrates various regions we
define below in relation to the unit square $U$ (recall that parts of
the barrier may be located in the exterior of $U$). 

\begin{itemize} \itemsep 1pt
\item [$\diamond$] $U=[0,1]^2$ is an (axis-aligned) unit square centered at
point $o=(1/2,1/2)$. 
\item [$\diamond$] $U_1=[w_1,1-w_1]^2$ is an (axis-aligned) square concentric with $U$. 
\item [$\diamond$] $U_2=[-w_2,1+w_2]^2$ is an (axis-aligned) square concentric with $U$. 
\item [$\diamond$] $U_3=[-1/2,3/2]^2$ is an (axis-aligned) square of
  side length $2$ concentric with $U$. 
\item [$\diamond$] $Q_1$ is a square of side length $\sqrt{2}/2$ concentric with
  $U$ and rotated by $\pi/4$. 
\item [$\diamond$] $Q_2$ is a square of side length $\sqrt{2}$ concentric with
  $U$ and rotated by $\pi/4$. 
\item [$\diamond$] $V = [0,1] \times (-\infty,+\infty)$
is the (infinite) vertical strip of unit width containing $U$.
\item [$\diamond$] $H = (-\infty,+\infty) \times [0,1]$
is the (infinite) horizontal strip of unit width containing $U$.
\item [$\diamond$] $d_1$ is $U$'s diagonal of positive slope and $d_2$ is $U$'s
  diagonal of negative slope. 
\item [$\diamond$] $W_1,W_2,W_3,W_4$ are the four wedges centered at $o$ and bounded by
the lines supporting the two diagonals of $U$, and directed to the
right, up, left, and down (\ie, in counterclockwise order).
\item [$\diamond$] $U_\textup{right}=[1-w_1,1+w_2] \times [0,1]$ is a
  thin rectangle of width $w_1+w_2$ and height $1$ whose right side
  coincides with the right side of $U_2$. Similarly, denote by 
$U_\textup{left}$, $U_\textup{low}$, $U_\textup{high}$, the
analogous rectangles contained in $U_2 \setminus U_1$ and sharing the corresponding
sides of $U_2$, as indicated. 
\end{itemize}

Observe that $U_1 \subset U \subset U_2$, and $Q_1 \subset U \subset Q_2$.
Note also that the inclusion 
$ U_1 \subset U_2 \setminus 
(U_\textup{low} \cup U_\textup{high} \cup U_\textup{left} \cup U_\textup{right})$
is strict. 
Complementary regions such as $\RR^2 \setminus U_2$,
$\RR^2 \setminus Q_2$, $\RR^2 \setminus V$, $\RR^2 \setminus H$,
are denoted by $\overline{U_2}$, $\overline{Q_2}$, $\overline{V}$,
$\overline{H}$.
\begin{figure}[htbp]
\centerline{\epsfxsize=5.1in \epsffile{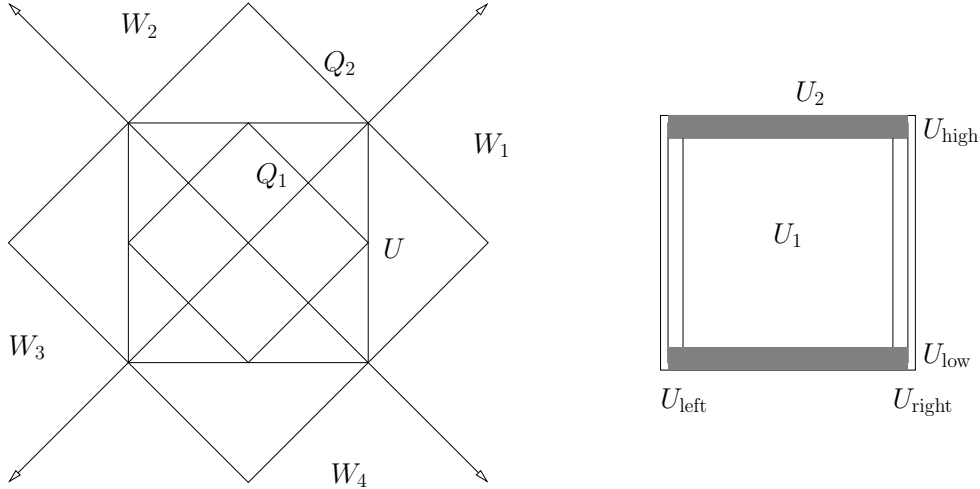}}
\caption{Left: The unit square $U$ and the rotated squares $Q_1$
and $Q_2$. Right: $U_1$ and $U_2$; two thin rectangles 
$U_\textup{low},U_\textup{high} \subset U_2 \setminus U_1$ (out of the
four) are shaded.}
\label{f3}
\end{figure}

We say that a segment (or a line) is {\em almost horizontal} if its
direction angle belongs to the interval $[-\phi,\phi]$.
Similarly, we say that a segment (or a line) is {\em almost vertical}
if its direction angle belongs to the interval 
$[\frac{\pi}{2}-\phi,\frac{\pi}{2}+\phi]$.
Let $\Gamma$ be a segment barrier for $U$ of length $L=|\Gamma|$. 
Let $X$ be the set of almost horizontal segments in $\Gamma$,
and $Y$ be the set of almost vertical segments in $\Gamma$. Let $Z$ be
the rest of the segments in $\Gamma$. Clearly, we have $L=|\Gamma|= |X|+|Y|+|Z|$.

\section{Local barriers: proof of Theorem~\ref{T1}} \label{sec:T1}

Let $\Gamma$ be a segment barrier for $U$ of length $L=|\Gamma|$. 
Without loss of generality by Lemma~\ref{L1}, we can assume that
$\Gamma$ is a segment barrier.  
Moreover, the lower bound of $2$ on its length (discussed previously) 
remains valid: $L \geq 2$. Assume for contradiction that $L \leq 2+\delta$. 
We first establish several structural properties of~$\Gamma$: 
\begin{itemize} \itemsep 1pt
\item $\Gamma$ must consist mostly of almost horizontal segments and
almost vertical segments. The total length of the almost horizontal
segments must be close to $1$, and similarly, the total length of the
almost vertical segments must be close to $1$. 
\item The total length of the segments in the exterior of $U_2$ must
  be small. 
\item For each side $s$ of $U$, a thin rectangle parallel to $s$ and
  enclosing $s$ must contain a set of significant weight consisting of barrier segments 
  almost parallel to $s$.
\end{itemize}

Once established, these structural properties of $\Gamma$ are used to find a line 
that is witness to the invalidity of $\Gamma$. By way of contradiction, 
the lower bound in Theorem~\ref{T1} will consequently follow. 
Let us record our initial assumptions to start with:
\begin{equation} \label{E4}
2 \leq L =|X|+|Y|+|Z| \leq 2+\delta.
\end{equation}

To begin our proof, we first refine Ozkan's argument  
(in Section~\ref{sec:prelim}) for the lower bound of $2$. 
We first show that the total length of the segments in $Z$ is small.

\begin{lemma}\label{L3}
The total length of the segments in $Z$ satisfies: 
$|Z| \leq 2 \cdot 10^8 \delta = 2 \sin \phi$.
\end{lemma}
\begin{proof}
Put $c=2 \cdot 10^8$, and assume for contradiction that $|Z| \geq c\delta$, hence
$|X|+|Y| \leq 2 +\delta -c\delta = 2 -(c-1)\delta$. 
Observe that for any segment in $Z$, we have (with $\theta_i$ as in
the respective proof) 
$$ |\cos \theta_i |+ |\sin \theta_i | \leq
\cos \left(\frac{\pi}{4}-\phi\right) +
\sin \left(\frac{\pi}{4}-\phi\right) :=a. $$
Note that 
$$ a= \frac{\sqrt2}{2} \cos \phi + \frac{\sqrt2}{2} \sin \phi +
\frac{\sqrt2}{2} \cos \phi - \frac{\sqrt2}{2} \sin \phi =
\sqrt2 \cos \phi <\sqrt{2}. $$
By the assumption, the first inequality in \eqref{E3} yields
$$ 2\sqrt{2} \leq (2 -(c-1)\delta) \sqrt2 + c a \delta < 2\sqrt{2}, $$
a contradiction.
Indeed, the second inequality in the above chain is equivalent to
$$ \cos \phi < \frac{c-1}{c}, $$
which holds since
\begin{equation*}
\cos \phi = \sqrt{1 - \sin^2 \phi} =
\sqrt{1 - 10^{-8}} < 1 - \frac{1}{2 \cdot 10^8} =\frac{c-1}{c}.
\tag*{\qedhere}
\end{equation*}
\end{proof}

Next we show that the total length of the almost horizontal segments
is close to $1$, and similarly,  that the total length of
the almost vertical segments is close to $1$.

\begin{lemma}\label{L4}
The following inequalities hold:
\begin{align} \label{E5}
1 - \frac72 \sin \phi &\leq |X| \leq 1 + \frac32 \sin \phi \nonumber \\
1 - \frac72 \sin \phi &\leq |Y| \leq 1 + \frac32 \sin \phi. 
\end{align}
\end{lemma}
\begin{proof}
It follows from~\eqref{E4} that $|X|+|Z| \leq 2+\delta -|Y|$.
We first prove the upper bounds: $|X| \leq 1 + \frac32 \sin \phi$ 
and $|Y| \leq 1 + \frac32 \sin \phi$. 
Assume first for contradiction that $|Y| \geq 1 + \frac32 \sin \phi$. 
The total projection length of the segments in $\Gamma$ on the
$x$-axis is at most
\begin{align*}
|X| + |Z| + |Y| \cos \left(\frac{\pi}{2}-\phi \right) &\leq
(2+ \delta -|Y|) +|Y| \sin \phi =
2+ \delta - |Y| (1-\sin \phi) \\
&\leq 2+ \delta - \left(1+\frac32 \sin \phi\right) (1-\sin \phi)=
1+ \delta -\frac12 \sin \phi + \frac32 \sin^2 \phi < 1, 
\end{align*}
\ie, smaller than the corresponding unit width of $U$.
This contradicts the opaqueness condition for vertical rays,
hence $|Y| \leq 1 + \frac32 \sin \phi$. Similarly we establish that 
$|X| \leq 1 + \frac32 \sin \phi$.

We now prove the lower bounds: $|X| \geq 1 - \frac72 \sin \phi$ and
$|Y| \geq 1 - \frac72 \sin \phi$.  
Assume  for contradiction that $|X| \leq 1 - \frac72 \sin \phi$.
Using the previous upper bound on $|Y|$, the total projection length of
the segments in $\Gamma$ on the $x$-axis is 
\begin{align*}
|X| + |Z| + |Y| \sin \phi &\leq
\left(1 - \frac72 \sin \phi\right)
 + 2 \sin \phi + \left(1 + \frac32 \sin \phi\right) \sin \phi \\
&= 1 - \frac12 \sin \phi + \frac32 \sin^2 \phi <1, 
\end{align*}
\ie, smaller than the corresponding width of $U$. 
This is a contradiction, hence $|X| \geq 1 - \frac72 \sin \phi$.
Similarly we establish that $|Y| \geq 1 -  \frac72 \sin \phi$.
\end{proof}

A further restriction on the placement of the segments in $\Gamma$ is given by:

\begin{lemma}\label{L5}
The following inequalities hold:
\begin{equation} \label{E6}
|X \cap \overline{V}| \leq \frac92 \sin \phi \quad
{\rm and } \quad |Y \cap \overline{H}| \leq \frac92 \sin \phi.
\end{equation}
\end{lemma}
\begin{proof}
By Lemma~\ref{L4}, we have $|Y| \geq 1 - \frac72 \sin \phi$.
Assume for contradiction that $|X \cap \overline{V}| \geq \frac92 \sin \phi$.
It follows that the total projection length of the segments in
$\Gamma$ on the interval $[0,1]$ of the $x$-axis is at most
\begin{align*}
|X \cap V| &+ |Z| + |Y| \sin \phi =
|X| - |X \cap \overline{V}| + |Z| + |Y| \sin \phi \\
&\leq 2+ \delta -\frac92 \sin \phi -|Y|  + |Y| \sin \phi 
=2+ \delta -\frac92 \sin \phi -|Y| (1-\sin \phi) \\
&\leq 2+ \delta -\frac92 \sin \phi - 
\left(1 - \frac72 \sin \phi \right) (1-\sin \phi) 
\leq 1+ \delta - \frac72 \sin^2 \phi <1, 
\end{align*}
\ie, smaller than the corresponding unit width.
This is a contradiction, hence $|X \cap \overline{V}| \leq \frac92 \sin \phi$.
Similarly we establish that $|Y \cap \overline{H}| \leq \frac92 \sin \phi$.
\end{proof}

\begin{lemma}\label{L6}
Let $I,J \subset \RR$ be two intervals $($not necessarily contained in $[0,1]$$)$. 
Then 
\begin{align*}
|X \cap I \times (-\infty,\infty)| &\leq |I \cap [0,1]| + 5 \sin \phi,
\textup{ and similarly } \\
|Y \cap (-\infty,\infty)| \times J &\leq |J \cap [0,1]| + 5 \sin \phi. 
\end{align*}
\end{lemma}
\begin{proof}
Put $\overline{I}=[0,1] \setminus I$, and $\overline{J}=[0,1] \setminus J$.
Assume for contradiction that 
$|X \cap I \times (-\infty,\infty)| \geq |I \cap [0,1]| + 5 \sin \phi$.
Then by Lemma~\ref{L5} we have
\begin{align*}
|X \cap \overline{I} \times (-\infty,\infty)| &= 
|X| -|X \cap I \times (-\infty,\infty)|  \\
&\leq \left(1 + \frac32 \sin \phi \right) -|I \cap [0,1]| -5 \sin \phi
= 1 - |I \cap [0,1]| -\frac72 \sin \phi.
\end{align*}
However, since
\begin{align*}
|X \cap \overline{I} \times (-\infty,\infty)| + |Z| + |Y| \sin \phi 
&\leq \left(1 - |I \cap [0,1]| -\frac72 \sin \phi \right) + 2 \sin \phi + 
\left(1 + \frac32 \sin \phi \right) \sin \phi \\
&= 1 - |I \cap [0,1]| -\frac12 \sin \phi + \frac32 \sin^2 \phi < 1 - |I \cap [0,1]|,
\end{align*}
the vertical lines intersecting the lower side of $U$ in
$\overline{I}$ are not blocked, which is a contradiction.

The proof of the second inequality is analogous.
\end{proof}

Next we show that the total length of the segments in $\Gamma$
lying in the exterior of $Q_2$ is small.

\begin{lemma}\label{L7}
The following inequality holds:
\begin{equation} \label{E7}
| \Gamma \cap \overline{Q_2} | \leq 4\delta.
\end{equation}
\end{lemma}
\begin{proof}
Assume for contradiction that
$| \Gamma \cap \overline{Q_2} | \geq 4\delta$.
Observe that any segment in $\Gamma \cap \overline{Q_2}$ projects
either in the exterior of $d_1$ on its supporting line, or in
the exterior of $d_2$ on its supporting line. It follows from
\eqref{E4} that the total length of the segments in $\Gamma$ that project
(at least in part) on both diagonals is at most $2 + \delta- 4\delta =
2 -3\delta$. Therefore the total projection length
of the segments in $\Gamma$ on the two diagonals (see also~\eqref{E3}) is
at most
$$ (2 -3\delta) \sqrt{2} + 4\delta = 2 \sqrt{2} 
- (3 \sqrt{2} -4) \delta < 2 \sqrt{2}, $$ 
that is, smaller than the sum of the lengths of the two diagonals.
This is a contradiction, hence $ |\Gamma \cap \overline{Q_2} | \leq 4\delta$.
\end{proof}

Next we show that the total length of the segments in $\Gamma$
lying in the exterior of $U_2$ is small.
We use again formulas established by Sylvester~\cite{Sy1890}; see
also~\cite[pp.~32--34]{Sa04} and the second proof for the bound of $2$
in Section~\ref{sec:prelim}. For a planar convex body
$K$, the measure of all lines that meet  $K$ is equal to $\per(K)$. 
In particular, the measure of all lines that meet a segment $s$ is equal to $2|s|$.  
Let now $K_1$, $K_2$ be two disjoint planar convex bodies and let
$L_1$ and $L_2$ be the lengths of the boundaries $\partial K_1$, $\partial K_2$.  
The \emph{external cover} $C_\textup{ext}$ of $K_1$ and $K_2$ is the boundary of 
$\conv(K_1 \cup K_2)$. 
The external cover may be interpreted as a closed elastic string
drawn about $K_1$ and $K_2$. Let $L_\textup{ext}$ denote the length of $C_\textup{ext}$.
The \emph{internal cover} $C_\textup{int}$ of $K_1$ and $K_2$ is the closed curve
realized by a closed elastic string drawn about $K_1$ and $K_2$ and
crossing over at a point between $K_1$ and $K_2$. 
Let $L_\textup{int}$ denote the length of $C_\textup{int}$.
Then, according to~\cite{Sy1890}, the measure of all lines that meet
$K_1$ and $K_2$ is $L_\textup{int} - L_\textup{ext}$. We need a technical lemma:

\begin{lemma} \label{L8}
Let $B$ be a convex body and let $s$ be a segment disjoint from $B$.
Let $\theta$ be the maximum angle of a minimum cone $C$ that contains $B$
and has apex $c$ in $s$, that is,
$\theta = \max_{c\in s} \min_{C\supseteq B} \angle C$.
Then the measure of all lines that meet both $B$ and $s$ is at most
$2\sin\frac{\theta}2 \cdot |s|$.
\end{lemma}
\begin{proof}
By Sylvester's formula, we only need to bound
$f(s,B) = L_{\mathrm{int}}(s,B) - L_{\mathrm{ext}}(s,B)$.
Since $f(s,B) = \int_s f(\mathrm{d}s,B)$,
we consider an arbitrarily short segment $\mathrm{d}s \subset s$.
Let $\gamma$ be the closed curve that is
the boundary of the convex hull of $\mathrm{d}s \cup B$.
Let $p$ and $q$ be the two endpoints of the subcurve of $\gamma$
consisting of points from the boundary of $B$ 
and let $u$ and $v$ be the two endpoints of $\mathrm{d}s$.
Refer to Figure~\ref{fig:measure}.
\begin{figure}[htb]
\centering\includegraphics[scale=0.99]{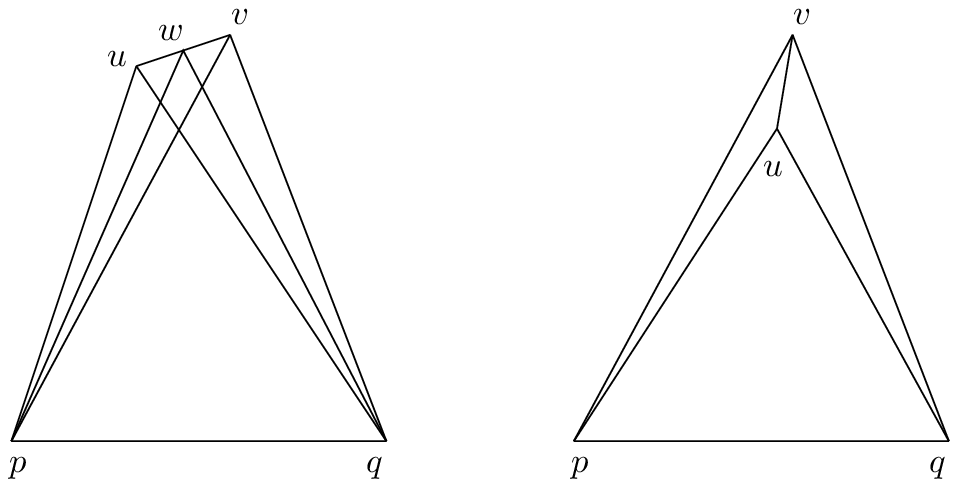}
\caption{Bounding $L_{\mathrm{int}}(s,B) - L_{\mathrm{ext}}(s,B)$.}
\label{fig:measure}
\end{figure}
We distinguish two cases:

\medskip
\emph{Case 1:} Both endpoints $u$ and $v$ are in $\gamma$.
Suppose the four points $p,u,v,q$ are in clockwise order along $\gamma$.
Then
$$
\lim_{|\mathrm{d}s|\to 0}
	\frac{f(\mathrm{d}s,B)}{\mathrm{d}s}
=
\lim_{|uv|\to 0}
	\frac{(|pv| - |pu|) + (|qu| - |qv|)}{|uv|}.
$$

We clearly have $\lim_{|uv|\to 0}  |uv|/|pv|=0$, and the Taylor expansion 
of $(1+x)^{1/2}$ around $0$, $(1+x)^{1/2}=1+x/2 +o(x)$, yields
\begin{align*}
\lim_{|uv|\to 0}
	\frac{|pv| - |pu|}{|uv|}
&=
\lim_{|uv|\to 0}
	\frac{|pv| - (|pv|^2 + |uv|^2 -
          2\cdot|pv|\cdot|uv|\cdot\cos\angle pvu)^{1/2}}{|uv|}\\ 
&=
\lim_{|uv|\to 0}
	\frac{|pv| - |pv|\cdot\big(1 + (|uv|/|pv|)^2 -
          (|uv|/|pv|)\cdot2\cos\angle pvu\big)^{1/2}}{|uv|}\\ 
&=
\lim_{|uv|\to 0}
	\frac{|pv| - |pv|\cdot\big(1 - (|uv|/|pv|)\cdot\cos\angle pvu\big)}{|uv|}\\
&= \cos\angle pvu.
\end{align*}

Symmetrically we have 
$$ \lim_{|uv|\to 0} \frac{|qu| - |qv|}{|uv|}= \cos\angle quv,$$
and thus 

$$
\lim_{|\mathrm{d}s|\to 0}
	\frac{f(\mathrm{d}s,B)}{\mathrm{d}s}
=
\lim_{|uv|\to 0}
	(\cos\angle pvu  + \cos\angle quv)
=
\lim_{|uv|\to 0}
	(\cos\angle pwu  + \cos\angle qwv)
$$
for any interior point $w$ of $\mathrm{d}s$.

Assume without loss of generality that $\angle pwu > \angle qwv$.
Put $\alpha = \frac{\angle pwu + \angle qwv}2$
and $\beta = \frac{\angle pwu - \angle qwv}2$.
Then $\angle pwu = \alpha + \beta$ and $\angle qwv = \alpha - \beta$.
It follows that
\begin{align*}
\cos\angle pwu + \cos\angle qwu
&= \cos(\alpha + \beta) + \cos(\alpha - \beta) \\
&= 2\cos\alpha\cos\beta
\le 2\cos\alpha = 2\cos\frac{\pi - \angle pwq}2 = 2\sin\frac{\angle pwq}2.
\end{align*}
Since $\angle pwq \le \theta$,
we have $\sin\frac{\angle pwq}2 \le \sin\frac{\theta}2$
and consequently
$$
\lim_{|\mathrm{d}s|\to 0}
	\frac{f(\mathrm{d}s,B)}{\mathrm{d}s}
= \lim_{|uv|\to 0} (\cos\angle pwu + \cos\angle qwv)
\le 2\sin\frac{\angle pwq}2
\le 2\sin\frac{\theta}2.
$$

\medskip
\emph{Case 2:} Only one endpoint of $s$, say $v$, is in $\gamma$.
Then
$$
\lim_{|\mathrm{d}s|\to 0}
	\frac{f(\mathrm{d}s,B)}{\mathrm{d}s}
=
\lim_{|uv|\to 0}
	\frac{(|pu| + |uv|  - |pv|) + (|qu| + |uv| - |qv|)}{|uv|}.
$$

The argument is analogous to that in case 1. Since
\begin{align*}
\lim_{|uv|\to 0}
	\frac{|pu| + |uv| - |pv|}{|uv|}
&=
\lim_{|uv|\to 0}
	\frac{(|pv|^2 + |uv|^2 - 2\cdot|pv|\cdot|uv|\cdot\cos\angle
          pvu)^{1/2} + |uv| - |pv|}{|uv|} 
\\
&=
\lim_{|uv|\to 0}
	\frac{|pv|\cdot\big(1 + (|uv|/|pv|)^2 -
          (|uv|/|pv|)\cdot2\cos\angle pvu\big)^{1/2} + |uv| - |pv|}{|uv|} 
\\
&=
\lim_{|uv|\to 0}
	\frac{|pv|\cdot\big(1 - (|uv|/|pv|)\cdot\cos\angle pvu\big) +
          |uv| - |pv|}{|uv|} \\
&= 1 - \cos\angle pvu,
\\
\intertext{and symmetrically}
\lim_{|uv|\to 0}
	\frac{|qu| + |uv| - |qv|}{|uv|}
&= 1 - \cos\angle qvu,
\end{align*}
we have
$$
\lim_{|\mathrm{d}s|\to 0}
	\frac{f(\mathrm{d}s,B)}{\mathrm{d}s}
= 2 - (\cos\angle pvu + \cos\angle qvu).
$$

Assume without loss of generality that $\angle pvu > \angle qvu$.
Put $\alpha = \frac{\angle pvu + \angle qvu}2$
and $\beta = \frac{\angle pvu - \angle qvu}2$.
Then $\angle pvu = \alpha + \beta$ and $\angle qvu = \alpha - \beta$.
It follows that
\begin{align*}
\cos\angle pvu + \cos\angle qvu
&= \cos(\alpha + \beta) + \cos(\alpha - \beta) \\
&= 2\cos\alpha\cos\beta
\ge 2\cos^2\alpha = 2\cos^2\frac{\angle pvq}2.
\end{align*}
Since $\angle pvq \le \theta$,
we have $\sin\frac{\angle pvq}2 \le \sin\frac{\theta}2$
and consequently
$$
\lim_{|\mathrm{d}s|\to 0}
	\frac{f(\mathrm{d}s,B)}{\mathrm{d}s}
= 2 - (\cos\angle pvu + \cos\angle qvu)
\le 2\sin^2\frac{\angle pvq}2
\le 2\sin^2\frac{\theta}2
\le 2\sin\frac{\theta}2.
$$

In summary, in both cases we have
$f(s,B) = \int_s f(\mathrm{d}s, B) \le 2\sin\frac{\theta}2 \cdot |s|$,
and the proof of Lemma~\ref{L8} is complete.
\end{proof}

\begin{corollary}\label{C1}
Consider a segment $s \in \Gamma \cap \overline{U_2}$. 
Then the measure of all lines that meet  $s$ and $U$ is 
at most $(1/4 + 10^{-6}) ^{-1/2} |s| < 2|s|$. 
\end{corollary}
\begin{proof}
Consider the setup in Lemma~\ref{L8}, with $B=U$. 
For $s \subset \overline{U_2}$, $\theta/2 < \pi/2$ is maximized when
the apex of the cone anchored in $s$ is the midpoint of one of the sides of $U_2$.
In this case we have 
$$ \sin\frac{\theta}2 = \frac{1}{2\sqrt{\frac14 + w_2^2}}=
\frac{1}{2\sqrt{\frac14 + 10^{-6}}}. $$
It follows from Lemma~\ref{L8} that the measure of all lines that meet
$s$ and $U$ is at most 
$2\sin\frac{\theta}2 \cdot |s| =(1/4 + 10^{-6}) ^{-1/2} |s|$, as required.  
\end{proof}

\begin{lemma}\label{L9}
The following inequality holds:
\begin{equation} \label{E8}
| \Gamma \cap \overline{U_2} | \leq (5 \cdot 10^5 +2) \, \delta 
=\frac{1}{200} \sin \phi + 2\delta.
\end{equation}
\end{lemma}
\begin{proof}
Let $\G$ denote all lines that meet $U$; 
let $\G_2$ denote all lines that meet some segment in $\Gamma \cap U_2$, and 
let $\G_{\overline{2}}$ denote all lines that meet $U$ and some segment in $\Gamma
\cap \overline{U_2}$. 

The measure of all lines that meet $U$ is equal to $m(\G)=\per(U)=4$. 
Since $m()$ is a measure, we have
\begin{align*}
4 &= m(\G) \leq m(\G_2) + m(\G_{\overline{2}}) \leq 
\sum_{s_i \in \Gamma \cap U_2} 2|s_i| + 
\sum_{s_i \in \Gamma \cap \overline{U_2}} (1/4 + 10^{-6}) ^{-1/2} |s_i| \\
&= 2 \sum_{s_i \in \Gamma} |s_i| - 
\sum_{s_i \in \Gamma \cap \overline{U_2}} (2-(1/4 + 10^{-6}) ^{-1/2}) |s_i| \\
&\leq 2(2+\delta) - 
\sum_{s_i \in \Gamma \cap \overline{U_2}} (2-(1/4 + 10^{-6}) ^{-1/2}) |s_i|.
\end{align*}
For the last inequality above we have used Lemma~\ref{L8}.

It follows that
\begin{equation*}
| \Gamma \cap \overline{U_2} | = \sum_{s_i \in \Gamma \cap \overline{U_2}} |s_i| 
\leq \frac{2}{2-(1/4 + 10^{-6}) ^{-1/2}} \, \delta 
\leq (5 \cdot 10^5 +2) \, \delta. 
\tag*{\qedhere}
\end{equation*}
\end{proof}

Refer to  Figure~\ref{f4}. 
Let $\ell'_+$ denote the line incident to the points $(19/20,0)$ and $(1,1/2)$;
let $\ell''_+$ denote the parallel line incident to the point $(1,0)$. 
Let $\Pi_+$ denote the parallel strip bounded by $\ell'_+$ and $\ell''_+$. 
Let $\Pi_-$ denote the parallel strip obtained by reflecting $\Pi_+$
about the horizontal line $y=1/2$. 
Recall that $w_1=1/20$. Observe that the slope of $\ell'_+$ is 
$\frac{1/2}{w_1} = \frac{1}{2w_1} = 10$, 
and let $\alpha$ denote the angle made by $\ell'_+$ with a vertical line.
Observe that $\tan \alpha= 2w_1 = 1/10 \gg \tan \phi \approx 10^{-4}$. 

\begin{figure}[htbp]
\centerline{\epsfxsize=5in \epsffile{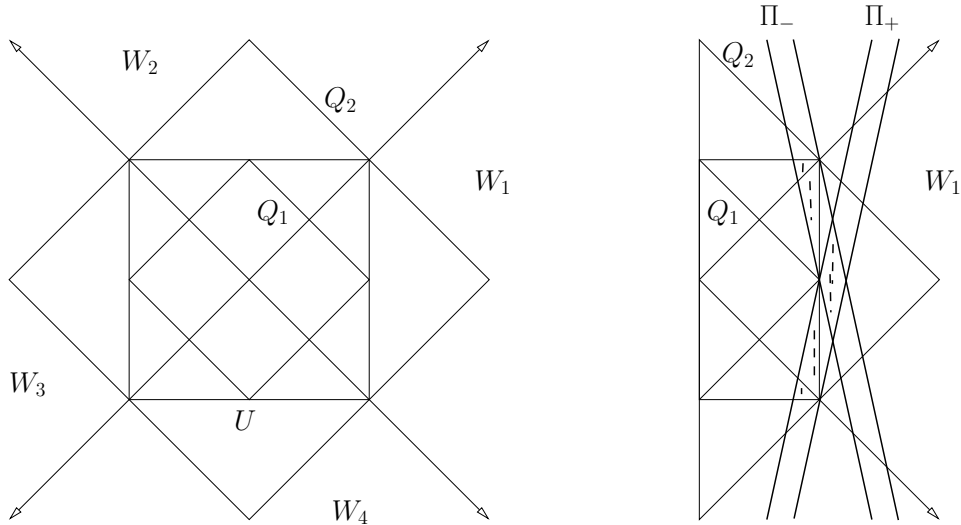}}
\caption{Opaqueness with respect to two parallel strips $\Pi_+$ and $\Pi_-$ 
(bounded by bold lines) covering the right side of $U$.}
\label{f4}
\end{figure}

\begin{lemma}\label{L10}
The following inequalities hold:
\begin{align*}
|X \cap U_\textup{low}| &\geq 0.45. \\
|X \cap U_\textup{high}| &\geq 0.45. \\
|Y \cap U_\textup{left}| &\geq 0.45. \\
|Y \cap U_\textup{right}| &\geq 0.45.
\end{align*}
\end{lemma}
\begin{proof}
We prove the 4th inequality; the proofs of the other three
inequalities are analogous. 
Put $w=w_1 +w_2/10$. Observe that the intersection point between
$\ell'_+$ and the lower side of $U_2$ is $(1-w,-w_2)$. 
Observe that the sets $X \cap V \cap U_2 \cap \Pi_+$, 
$X \cap V \cap U_2 \cap \Pi_-$, and $X \cap \overline{V}$ are pairwise
disjoint because the regions $V \cap \Pi_+$,  $V \cap \Pi_-$, and
$\overline{V}$ are pairwise disjoint. Let 
$x_+ = |X \cap V \cap U_2 \cap \Pi_+|$, 
$x_- = |X \cap V \cap U_2 \cap \Pi_-|$, and 
$x_* = |X \cap \overline{V}|$. 
Setting $I=[1-w,1]$ in Lemma~\ref{L6} yields
\begin{equation} \label{E9}
x_+ + x_- \leq w + 5 \sin \phi.
\end{equation}
According to Lemma~\ref{L5}, 
\begin{equation} \label{E10}
x_* \leq \frac92 \sin \phi.
\end{equation}

By~\eqref{E9} and~\eqref{E10} we further have
\begin{equation} \label{E11}
x_+ + x_- + 2x_* \leq w+(5+9) \sin \phi = w + 14 \sin \phi.
\end{equation}

Note that $\Pi_+ \cap \Pi_- \subset H$. Observe that the sets 
$Y \cap H \cap U_2 \cap (\Pi_+ \setminus \Pi_-)$,
$Y \cap H \cap U_2 \cap (\Pi_- \setminus \Pi_+)$, 
$Y \cap U_2 \cap (\Pi_+ \cap \Pi_-)$, and
$Y \cap (\overline{H} \cup \overline{U_2})$ are pairwise
disjoint because the regions 
$H \cap U_2 \cap (\Pi_+ \setminus \Pi_-)$,
$H \cap U_2 \cap (\Pi_- \setminus \Pi_+)$, 
$U_2 \cap (\Pi_+ \cap \Pi_-)$, and
$\overline{H} \cup \overline{U_2}$ are pairwise disjoint. Let 
$y_+ = |Y \cap H \cap U_2 \cap (\Pi_+ \setminus \Pi_-)|$,
$y_- = |Y \cap H \cap U_2 \cap (\Pi_- \setminus \Pi_+)|$,
$y_\pm = |Y \cap U_2 \cap H \cap (\Pi_+ \cap \Pi_-)|$, and
$y_* = |Y \cap (\overline{H} \cup \overline{U_2})|$. 
Note also that
\begin{equation} \label{E12}
[H \cap U_2 \cap (\Pi_+ \setminus \Pi_-)] \cup
[H \cap U_2 \cap (\Pi_- \setminus \Pi_+)] \cup 
[U_2 \cap (\Pi_+ \cap \Pi_-)] \subset U_\textup{right}. 
\end{equation}

Let  $J=[\frac12 - \frac{w_2}{2 w_1}, \frac12 + \frac{w_2}{2 w_1}]$.
Note that $\Pi_+ \cap \Pi_- \cap U_2$ forms an isosceles triangle with
longer vertical side $[1+ w_2,1+w_2] \times J$. 
Setting $J=[\frac12 - \frac{w_2}{2 w_1}, \frac12 + \frac{w_2}{2 w_1}]$ 
in Lemma~\ref{L6} yields
\begin{equation} \label{E13}
y_\pm \leq |J| + 5 \sin \phi = \frac{w_2}{w_1} + 5 \sin \phi 
= \frac{1}{50} + 5 \sin \phi. 
\end{equation}
By Lemma~\ref{L5} and Lemma~\ref{L9} we have
\begin{equation} \label{E14}
y_* \leq \frac92 \sin \phi +  \left(\frac{1}{200} \sin \phi + 2 \delta \right)
\leq 5 \sin \phi.
\end{equation}

The opaqueness condition for the lines in $\Pi_+$ implies that
$$ (x_+ + x_*) \cos(\alpha-\phi) + (y_+ + y_\pm + y_*) \cos(\pi/2 -\alpha-\phi) + |Z| 
\geq w_1 \cos \alpha. $$
Similarly, the opaqueness condition for the lines in $\Pi_-$ implies that
$$ (x_- + x_*) \cos(\alpha-\phi) + (y_- + y_\pm + y_*) \cos(\pi/2 -\alpha-\phi) + |Z| 
\geq w_1 \cos \alpha. $$
Adding these two inequalities yields
$$ (x_+ + x_- + 2x_*) \cos(\alpha-\phi) 
+ (y_+ + y_- + 2y_\pm + 2y_*) \sin(\alpha+\phi) + 2|Z| 
\geq 2 w_1 \cos \alpha. $$ 
After dividing by $\sin (\alpha + \phi)$ we get
$$ \frac{\cos(\alpha-\phi)}{\sin(\alpha+\phi)} \, (x_+ + x_- + 2x_*) 
+ (y_+ + y_- + 2y_\pm + 2y_*) 
+ \frac{2|Z|}{\sin(\alpha+\phi)} \geq 
\frac{\sin \alpha}{\sin (\alpha + \phi)}, \textup{ or }$$
\begin{equation} \label{E15}
y_+ + y_- + y_\pm \geq \frac{\sin \alpha}{\sin (\alpha + \phi)} 
- \frac{\cos(\alpha-\phi)}{\sin(\alpha+\phi)} \, (x_+ + x_- + 2x_*)
- \frac{2|Z|}{\sin(\alpha+\phi)}  -y_\pm -2y_*. 
\end{equation}
Observe that the following rough ideal approximations 
\begin{align*}
\frac{\sin \alpha}{\sin (\alpha + \phi)} &\approx 1, \ \ 
\frac{\cos(\alpha-\phi)}{\sin(\alpha+\phi)} \approx \frac{1}{\tan \alpha} = \frac{1}{2w_1}, \ \ 
x_+ + x_- + 2x_* \lessapprox w, \\
\frac{2|Z|}{\sin(\alpha+\phi)} &\approx 0, \ \ 
y_\pm \lessapprox \frac{w_2}{w_1}, \ \ 
y_* \approx 0,
\end{align*}
would yield 
\begin{align*}
y_+ + y_- + y_\pm &\gtrapprox 1- \frac{w}{2w_1} -\frac{w_2}{w_1} 
= 1- \frac{w_1 + w_2/10}{2w_1} -\frac{w_2}{w_1} \\
&= 1- \frac12 - \frac{21 w_2}{20 w_1} 
= \frac12 -\frac{21 \cdot 20}{20 \cdot 1000}
= \frac{479}{1000}. 
\end{align*}

In reality we have the slightly weaker bounds:
\begin{align*}
\frac{\sin \alpha}{\sin (\alpha + \phi)} &\geq 0.999, \ \ 
\frac{\cos(\alpha-\phi)}{\sin(\alpha+\phi)} \, (x_+ + x_- + 2x_*) 
\leq 10 \left(w + 14 \sin \phi \right), \\
\frac{2|Z|}{\sin(\alpha+\phi)}  &\leq 40 \sin \phi, \ \ 
y_\pm \leq \frac{w_2}{w_1} + 5 \sin \phi, \ \ 
y_* \leq 5 \sin \phi + 2 \delta.
\end{align*}

So instead, by taking into account these bounds, 
inequality~\eqref{E15} implies that:
\begin{equation} \label{E16}
y_+ + y_- + y_\pm 
\geq \frac{478}{1000} - 140 \sin \phi - 40 \sin \phi  - 15 \sin \phi -4\delta 
\geq \frac{45}{100}. 
\end{equation}

Taking into account~\eqref{E12}, the 4th inequality in the lemma follows:  
$$ |Y \cap U_\textup{right}| 
\geq y_+ + y_- + y_\pm \geq 0.45. $$
The proof of Lemma~\ref{L10} is now complete.
\end{proof}

\paragraph{Last step in the proof.} We describe a procedure {\sc ADVANCE} 
for finding a line that is witness to the invalidity of
$\Gamma$ with respect to $U$. For convenience we refer to such an event as
terminating the procedure with \emph{success}. According to our assumption for
$\Gamma$ being a barrier, if this happens, it will be an obvious contradiction.
The analysis of {\sc ADVANCE} employs a potential argument ultimately
based on the inequalities established in Lemma~\ref{L10}.  
The fact that $\Gamma$ lies is $U_3$ is key to the analysis of {\sc ADVANCE}, 
and thus to our proof. 

The initial position of the sweep-line $\ell$ is the vertical line $x=w_1$. 
While $\ell$ is infinite, it is convenient to view it as anchored at its
intersection points with the two horizontal sides of $U_3$:
$a_\textup{low}$ on the lower side and $a_\textup{high}$ on the higher
side; the two anchor points change as $\ell$ changes its position.  
The line $\ell$ moves right across the central part of $U$ (resp., $U_2$), 
in the sense that its anchor points are stationary or move to
the right on the corresponding sides of $U_3$, as follows. See fig.~\ref{f5}. 
\begin{enumerate}
\item If $\ell$ intersects segments in $X \cap U_\textup{high}$, 
  then $\ell$ rotates clockwise around  $a_\textup{low}$ until this
  condition fails.  
If $\ell$ intersects segments in $X \cap U_\textup{low}$,
  then  $\ell$ rotates counterclockwise around  $a_\textup{high}$ 
  until this condition fails. 
(If $\ell$ does not intersect segments in $X \cap U_\textup{low}$ or
$X \cap U_\textup{high}$, rule 2 applies.)
\item If $\ell$ intersects other segments of $\Gamma$, 
then $\ell$ moves right remaining parallel
  to itself until this condition fails. The two anchor points
  $a_\textup{low}$ and $a_\textup{high}$ move right by the same amount
  on the corresponding sides of $U_3$.
\end{enumerate}
\begin{figure}[htbp]
\centerline{\epsfxsize=3.7in \epsffile{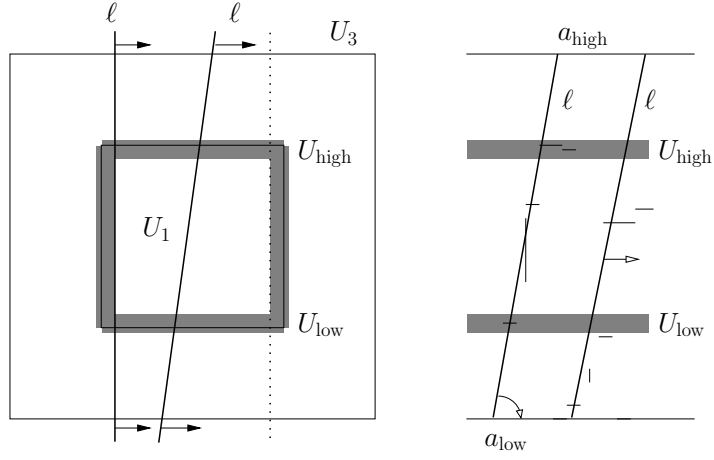}}
\caption{Left: the sweep-line in procedure {\sc ADVANCE} moving right across the
  central part of $U$ (the central subrectangles of 
$U_\textup{low},U_\textup{high} \subset U_2 \setminus U_1$).
Right: two cases (rotation and translation) of charges with the sweep-line.}
\label{f5}
\end{figure}

We next show that {\sc ADVANCE} achieves success before any of its
two anchor points reaches the vertical line $x=1-w_1$ (supporting
the left side of $U_\textup{right}$).  

\begin{lemma}\label{L11} The following properties hold:
\begin{itemize} \itemsep 1pt

\item [{\rm (i)}] During the execution of {\sc ADVANCE}, the slope of $\ell$ 
in absolute value is at least 
$$\tan \beta = \frac{3/2 -w_1 - x_1 \sin \phi}{x_1 \cos \phi} \geq 2.635, $$
where $x_1=|X| -0.45$. 

\item [{\rm (ii)}] The total advance of the higher anchor point caused
by rotations of $\ell$ $($sweeping over segments in $X \cap U_\textup{high}$$)$ 
around the lower anchor point is at most $x_3 = \frac{2}{\tan \beta} \leq 0.76$. 

\item [{\rm (iii)}] The total advance of an anchor anchor point caused
by translations of $\ell$ over segments in $X$ is at most
$$ \frac{\sin (\beta+\phi)}{\sin \beta} x_4 \leq 1.01 x_4 , $$
where $x_4 = |X | - |X \cap U_\textup{low}| - |X \cap U_\textup{high}|$. 

\item [{\rm (iv)}] The total advance of an anchor anchor point caused
by translations of $\ell$ over segments in $Y$ is at most
$$\frac{\cos (\beta -\phi)}{\sin \beta} y_1 \leq 0.38 y_1, $$
where $y_1 = |Y| -  |Y \cap U_\textup{left}| - |Y \cap U_\textup{right}|$. 

\item [{\rm (v)}] The total advance of an anchor anchor point caused
by translations of $\ell$ over segments in $Z$ is at most
$$ \frac{|Z|}{\sin \beta} \leq 1.1 |Z|. $$
\end{itemize}
\end{lemma}
\begin{proof}
For simplicity, throughout this proof, we denote a segment and its
length by the same letter when there is no danger of confusion. 

(i) Observe that the slope of $\ell$ can change only during rotation. 
Each rotation is attributed to some segment in $X \cap U_\textup{high}$
or to some segment in $X \cap U_\textup{low}$. We make the argument
for the first case; during rotation the anchor point is fixed on the
lower side of $U_3$. Since the initial position of $\ell$ is vertical,
the slope of $\ell$ cannot exceed the slope of the hypotenuse of the
right triangle in Figure~\ref{f6}, where all segments in $X \cap U_\textup{high}$ 
have been concatenated and made collinear in the segment $x_1$
which makes an angle of $\phi$ with a horizontal line. We now
determine the slope of the hypotenuse. 
\begin{figure}[htbp]
\centerline{\epsfxsize=3.5in \epsffile{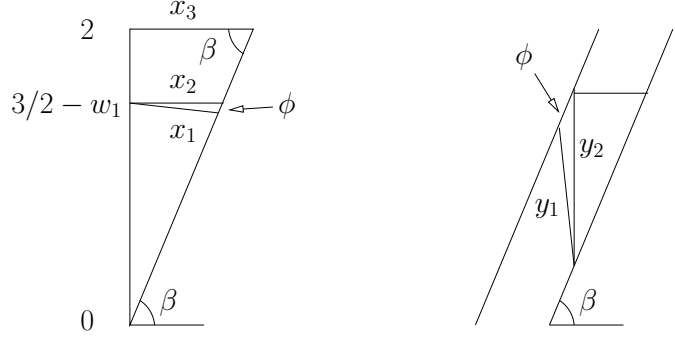}}
\caption{Left: bounding the advance of the higher anchor point caused by
  rotations over segments in $X \cap U_\textup{high}$. 
Right: bounding the advance of both anchor points caused by translations
over segments in $Y$.}
\label{f6}
\end{figure}
Let $x_2$ be the horizontal segment incident to the higher endpoint of
$x_1$.   
By~Lemma~\ref{L4}, $x_1 \leq 1 + 1.5 \sin \phi -0.45 = 0.55 + \sin \phi$. 
By the law of sines in the small triangle with sides $x_1$ and $x_2$, we have 
\begin{equation} \label{E17}
\frac{x_1}{\sin \beta} = \frac{x_2}{\sin (\beta+\phi)}
\ \ \ \Rightarrow \ \ \ 
x_2 = \left( \cos \phi + \frac{\sin \phi}{\tan \beta} \right)\, x_1. 
\end{equation}
We also have 
\begin{equation} \label{E18}
\tan \beta = \frac{3/2-w_1}{x_2} 
\ \ \  \Rightarrow \ \ \ 
x_2 = \frac{3/2-w_1}{\tan \beta}. 
\end{equation}
Putting~\eqref{E17} and~\eqref{E18} together yields
\begin{equation} \label{E19}
\tan \beta = \frac{3/2 -w_1 - x_1 \sin \phi}{x_1 \cos \phi} \geq
\frac{1.5 -w_1}{(0.55 + 1.5 \sin \phi) \cos \phi} - \tan \phi 
\geq 2.635. 
\end{equation}

\medskip
(ii)
Refer to Figure~\ref{f6}~(left).
Clearly, the total advance of the higher anchor point
is at most $ x_3 = \frac{2}{\tan \beta} \leq 0.76$.

\medskip
(iii) As in (i),  the maximum advance is achieved when the slope of
$\ell$ is the smallest in absolute value and all segments in $X$ make
an angle of $\phi$ clockwise below the horizontal line. 
The total length of segments in $X$ contributing to translations of
$\ell$ is clearly bounded from above by 
$x_4 = |X|- |X \cap U_\textup{low}| - |X \cap U_\textup{high}|$.
As in~\eqref{E17}, the total advance of each anchor point is at most  
$\frac{\sin (\beta+\phi)}{\sin \beta} \, x_4 \le 1.01 x_4$. 

\medskip
(iv) Refer to Figure~\ref{f6}~(right). The maximum advance is
achieved when the slope of $\ell$ is the smallest in absolute value. 
We can assume that all segments swept over are collinear in a segment 
$y_1$ that makes an angle of $\phi$ with the vertical direction, as
shown in the figure; here $y_2$ is a vertical segment sharing an
endpoint with $y_1$. 
By the law of sines in the small triangle with sides $y_1$ and $y_2$, we have   
\begin{equation} \label{E20}
\frac{y_1}{\sin (\pi/2 -\beta)} = \frac{y_2}{\sin (\phi+ \pi/2 -\beta)}
\ \ \Rightarrow \ \ 
y_2 = \frac{\cos (\beta -\phi)}{\sin \beta} y_1 \leq 0.38 y_1, 
\end{equation}
as required. 

\medskip
(v) Similarly with (iv), we deduce that the advance is at most
\begin{equation} \label{E21}
\frac{|Z|}{\sin \beta} = \sqrt{1 + \frac{1}{\tan^2 \beta}} \, |Z| 
\leq 1.1 \, |Z|, 
\end{equation}
as required. 

\medskip
The proof of Lemma~\ref{L11} is now complete.
\end{proof}

To finish the proof of Theorem~\ref{T1}, we next bound from above the
total advance of each anchor point.  Note that 
$$ x_4 = |X|- |X \cap U_\textup{low}| - |X \cap U_\textup{high}| 
\leq \left( 1+ \frac32 \sin \phi \right) -0.45 - 0.45
\leq 0.1 + \frac32 \sin \phi, $$ 
and similarly
$$ y_1 = |Y| - |Y \cap U_\textup{left}| - |Y \cap U_\textup{right}| \leq
\left( 1+ \frac32 \sin \phi \right) -0.45 - 0.45 \leq 0.1 + \frac32 \sin \phi. $$
By Lemma~\ref{L11}~(ii), the total advance of an anchor point due to
rotations of $\ell$ is at most $0.76$.  

Therefore, taking into account the inequalities in Lemma~\ref{L11},
the total advance of an anchor point is at most 
$$ 0.76 + 1.01 x_4 +  0.38 y_1 + 1.1 |Z| 
\leq  0.76 + 0.1013 + 0.0381 + 2.2 \sin \phi \leq 0.8997 < 0.9, $$
\ie, strictly smaller than the horizontal distance of $1-2w_1=0.9$ between 
the right side of $U_\textup{left}$ and the left side of $U_\textup{right}$. 

Consequently, the execution of procedure {\sc ADVANCE} terminates with
success and this concludes the proof of Theorem~\ref{T1}.
\qed

\paragraph{Remark.} The reader may wonder where the assumption 
$\Gamma \subset U_3$ was needed. If $\Gamma$ is not confined to $U_3$,
since $Q_2 \subset U_3$, 
we know by Lemma~\ref{L7} that the total length of its segments located in the
exterior of $U_3$ is small. However, these segments can pose
difficulty in the analysis of {\sc ADVANCE} because they could be
swept by the sweep-line multiple times, backward (in the ``wrong''
direction) during rotation, and then forward (in the ``correct''
direction) during translation, and then again backward and forward, etc.

\section{A sharper bound for interior barriers by linear programming} 
\label{sec:T2}

In this section we prove Theorem~\ref{T2}, 
namely that the length of any interior barrier for
the unit square is at least $2 + 10^{-5}$.
Let $w$ be a small number to be determined, $0 < w < 1/2$.
Put $\psi = \arctan 2w$.
Let $\phi$ be a small angle to be determined, $0 < \phi < \psi$.
(We will set $w = 0.1793$ and $\phi = 1.5589^\circ$.)
We say that a segment $s$ is \emph{near horizontal}
(resp.~\emph{near vertical})
if the angle between the segment and the $x$-axis
(resp.~$y$-axis) is at most $\phi$.
Refer to Figure~\ref{fig:square}.

Divide the unit square $U = [0,1]^2$ into $13$ convex sub-regions 
(one octagon, eight triangles and four quadrilaterals) by $8$ segments,
each cutting off a right triangle with two shorter sides of lengths 
$w$ and $1/2$. 
The height of each triangle to its hypotenuse is
$h = \frac{w/2}{\sqrt{w^2 + 1/4}} = \frac{1}{\sqrt{4 + w^{-2}}}$.
This partition of $U$ is suggested by our earlier Lemma~\ref{L10}.
\begin{figure}[htbp]
\centering\includegraphics{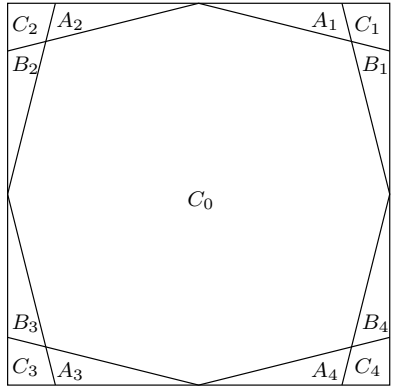}
\caption{Partition the unit square $U$ into $13$ parts.}
\label{fig:square}
\end{figure}

Let $\Gamma$ be an interior segment barrier for $U$.
Let $X$ (resp.~$Y$)
be the subset of $\Gamma$ consisting of near horizontal
(resp.~near vertical) segments.
Let $Z = \Gamma \setminus (X \cup Y)$.
Partition each of $X$, $Y$, and $Z$ further into $13$ subsets
consisting of segments within the $13$ sub-regions, respectively.
We thereby obtain a partition of $\Gamma$ into $39$ subsets.
In the following, we construct a linear program with $39$ variables,
one variable for the total length of segments in each subset,
and with the goal of minimizing the sum of the $39$ variables.

\subsection{Linear constraints based on opaque conditions of projections}

For each segment $s$ in $\Gamma$,
denote by $\alpha_s$ the smallest angle of rotation
that brings $s$ to either horizontal or vertical,
and
denote by $\beta_s$ the smallest angle between $s$ and a diagonal of $U$.
Then,
$0 \le \alpha_s \le \pi/4$,
$0 \le \beta_s \le \pi/4$,
and
$\alpha_s + \beta_s = \pi/4$.
Denote by $|s|_x$ (resp.~$|s|_y$) the length of projection of $s$ to a
horizontal (resp.~vertical) side of $U$.
Let $|s|_{xy} = |s|_x + |s|_y$.
Denote by $|s|_{zz}$ the total length of projection of $s$
to the two diagonals of $U$.
Clearly,
\begin{align*}
|s|_{xy} &= |s|\cdot (\cos\alpha_s + \sin\alpha_s)
	= |s|\cdot \sqrt2\cos\beta_s
\\
|s|_{zz} &= |s|\cdot (\cos\beta_s + \sin\beta_s)
	= |s|\cdot \sqrt2\cos\alpha_s.
\end{align*}

The total length of a horizontal side and a vertical side of $U$ is $2$.
The opaque conditions in the horizontal direction and the vertical direction
require that
\begin{equation}\label{eq:sides}
\sum_{s\in X\cup Y} |s|_{xy} + \sum_{s\in Z} |s|_{xy} \ge 2.
\end{equation}

The total length of the two diagonals of $U$ is $2\sqrt2$.
The opaque conditions in the two directions perpendicular to the two diagonals
require that
\begin{equation}\label{eq:diagonals}
\sum_{s\in X\cup Y} |s|_{zz} + \sum_{s\in Z} |s|_{zz} \ge 2\sqrt2.
\end{equation}

For each segment $s$ in $X$ or $Y$, we have
$\alpha_s \in [0, \phi]$
and
$\beta_s \in [\frac{\pi}4 - \phi, \frac{\pi}4]$.
Thus
\begin{align*}
|s| &\le |s|_{xy} \le |s|\cdot\sqrt2\cos\left( \frac{\pi}4 - \phi \right),
\\
|s|\cdot\sqrt2\cos\phi &\le |s|_{zz} \le |s|\cdot\sqrt2.
\\
\intertext{\indent For each segment $s$ in $Z$, we have
$\alpha_s \in (\phi, \frac{\pi}4]$
and
$\beta_s \in [0, \frac{\pi}4 - \phi)$.
Thus}
|s|\cdot\sqrt2\cos\left( \frac{\pi}4 - \phi \right)
	&< |s|_{xy} \le |s|\cdot\sqrt2,
\\
|s| &\le |s|_{zz} < |s|\cdot\sqrt2\cos\phi.
\end{align*}

Using the upper bounds in the inequalities above,
it follows from~\eqref{eq:sides} and~\eqref{eq:diagonals} that
\begin{equation}\label{eq:xy}
|X \cup Y|\cdot\sqrt2\cos\left( \frac{\pi}4 - \phi \right)
	+ |Z|\cdot\sqrt2 \ge 2,
\end{equation}
\begin{equation}\label{eq:zz}
|X \cup Y|\cdot\sqrt2
+ |Z|\cdot\sqrt2\cos\phi \ge 2\sqrt2.
\end{equation}

\paragraph{More projections to the sides.}
For each subset $S$ of $U$, let $\overline{S}$ denote $U \setminus S$.

\old{
The opaque conditions in the directions perpendicular to the $x$-axis
and the $y$-axis, respectively, require that
\begin{equation}\label{eq:x}
|X| + |Y|\cdot\sin\phi + |Z|\cdot\cos\phi \ge 1,
\end{equation}
\begin{equation}\label{eq:y}
|Y| + |X|\cdot\sin\phi + |Z|\cdot\cos\phi \ge 1.
\end{equation}

For $S = C_0 \cup A_1 \cup B_1 \cup C_1 \cup A_4 \cup B_4 \cup C_4$
and $S = C_0 \cup A_2 \cup B_2 \cup C_2 \cup A_3 \cup B_3 \cup C_3$,
\begin{equation}\label{eq:x+ii}
  |X \cap S|
+ |Y \cap S|\cdot\sin\phi
+ |Z \cap S|\cdot\cos\phi
\ge \frac12.
\end{equation}

For $S = C_0 \cup A_1 \cup B_1 \cup C_1 \cup A_2 \cup B_2 \cup C_2$
and $S = C_0 \cup A_3 \cup B_3 \cup C_3 \cup A_4 \cup B_4 \cup C_4$,
\begin{equation}\label{eq:y+ii}
  |Y \cap S|
+ |X \cap S|\cdot\sin\phi
+ |Z \cap S|\cdot\cos\phi
\ge \frac12.
\end{equation}
}

For $S = C_0 \cup A_1 \cup A_2 \cup A_3 \cup A_4$,
\begin{equation}\label{eq:x+a}
  |X \cap S|
+ |Y \cap S|\cdot\sin\phi
+ |Z \cap S|\cdot\cos\phi
\ge 1 - 2w.
\end{equation}

For $S = C_0 \cup B_1 \cup B_2 \cup B_3 \cup B_4$,
\begin{equation}\label{eq:y+b}
  |Y \cap S|
+ |X \cap S|\cdot\sin\phi
+ |Z \cap S|\cdot\cos\phi
\ge 1 - 2w.
\end{equation}

For $S = C_0 \cup A_1 \cup A_4$ and $S = C_0 \cup A_2 \cup A_3$,
\begin{equation}\label{eq:x+aii}
  |X \cap S|
+ |Y \cap S|\cdot\sin\phi
+ |Z \cap S|\cdot\cos\phi
\ge \frac12 - w.
\end{equation}

For $S = C_0 \cup B_1 \cup B_2$ and $C_0 \cup B_3 \cup B_4$,
\begin{equation}\label{eq:y+bii}
  |Y \cap S|
+ |X \cap S|\cdot\sin\phi
+ |Z \cap S|\cdot\cos\phi
\ge \frac12 - w.
\end{equation}

For $S = \overline{B_1 \cup C_1 \cup B_4 \cup C_4}$
and $S = \overline{B_2 \cup C_2 \cup B_3 \cup C_3}$,
\begin{equation}\label{eq:x-bcii}
  |X \cap S|
+ |Y \cap S|\cdot\sin\phi
+ |Z \cap S|\cdot\cos\phi
\ge 1 - w.
\end{equation}

For $S = \overline{A_1 \cup C_1 \cup A_2 \cup C_2}$
and $S = \overline{A_3 \cup C_3 \cup A_4 \cup C_4}$,
\begin{equation}\label{eq:y-acii}
  |Y \cap S|
+ |X \cap S|\cdot\sin\phi
+ |Z \cap S|\cdot\cos\phi
\ge 1 - w.
\end{equation}

\paragraph{More projections to the diagonals.}
For $S = C_0 \cup A_i \cup B_i \cup C_i$, $1 \le i \le 4$,
\begin{equation}\label{eq:z+i}
|(X \cup Y) \cap S|\cdot\cos\left( \frac{\pi}4 - \phi \right)
	+ |Z \cap S|
\ge \frac14\sqrt2.
\end{equation}

For $S = \overline{A_i \cup B_i \cup C_i}$, $1 \le i \le 4$,
\begin{equation}\label{eq:z-i}
|(X \cup Y) \cap S|\cdot\cos\left( \frac{\pi}4 - \phi \right)
	+ |Z \cap S|
\ge \frac34\sqrt2.
\end{equation}

\old{
For $S = C_0 \cup A_1 \cup B_1 \cup C_1 \cup A_3 \cup B_3 \cup C_3$
and $S = C_0 \cup A_2 \cup B_2 \cup C_2 \cup A_4 \cup B_4 \cup C_4$,
\begin{equation}\label{eq:z+ii}
|(X \cup Y) \cap S|\cdot\cos\left( \frac{\pi}4 - \phi \right)
	+ |Z \cap S|
\ge \frac12\sqrt2.
\end{equation}
}

\paragraph{Projections along the hypotenuses.}
For $S = B_i \cup C_i$, $1 \le i \le 4$,
\begin{equation}\label{eq:cb+i}
  |X \cap S|\cdot\cos(\psi - \phi)
+ |Y \cap S|\cdot\sin(\psi + \phi)
+ |Z \cap S|
\ge h.
\end{equation}

For $S = A_i \cup C_i$, $1 \le i \le 4$,
\begin{equation}\label{eq:ca+i}
  |Y \cap S|\cdot\cos(\psi - \phi)
+ |X \cap S|\cdot\sin(\psi + \phi)
+ |Z \cap S|
\ge h.
\end{equation}

\old{
The distance between any two parallel hypotenuses of right triangles is
$d = h/w$.

For $S = \overline{B_1 \cup C_1 \cup B_3 \cup C_3}$
and $S = \overline{B_2 \cup C_2 \cup B_4 \cup C_4}$,
\begin{equation}\label{eq:cb-ii}
  |X \cap S|\cdot\cos(\psi - \phi)
+ |Y \cap S|\cdot\sin(\psi + \phi)
+ |Z \cap S|
\ge d.
\end{equation}

For $S = \overline{A_1 \cup C_1 \cup A_3 \cup C_3}$
and $S = \overline{A_2 \cup C_2 \cup A_4 \cup C_4}$,
\begin{equation}\label{eq:ca-ii}
  |Y \cap S|\cdot\cos(\psi - \phi)
+ |X \cap S|\cdot\sin(\psi + \phi)
+ |Z \cap S|
\ge d.
\end{equation}
}

\subsection{Linear constraints based on the ADVANCE procedure}

Now consider the ADVANCE procedure with
$a_{\mathrm{high}}$ and $a_{\mathrm{low}}$
on the upper and lower sides of $U$, respectively,
with $x$ coordinates between $w$ and $1 - w$.
Put $\beta = \arctan\frac1{1-2w}$.
Recall Lemma~\ref{L11}.
By a similar analysis as in items (i) and (ii) of Lemma~\ref{L11},
each segment in $X \cap (A_1 \cup A_2)$
causes a rotation that moves $a_{\mathrm{high}}$
for a distance at most its length times
the factor $\frac{\sin(\beta+\phi)}{\sin\beta} \cdot \frac{1}{1-w}$,
and each segment in $X \cap (A_3 \cup A_4)$
causes an analogous movement of $a_{\mathrm{low}}$.
Also, as in (iii), (iv), and (v) of Lemma~\ref{L11},
each segment in $X \cap C_0$,
each segment in $Y \cap (C_0 \cup A_1 \cup A_2 \cup A_3 \cup A_4)$, and
each segment in $Z \cap (C_0 \cup A_1 \cup A_2 \cup A_3 \cup A_4)$,
respectively,
causes a translation that moves both $a_{\mathrm{high}}$ and $a_{\mathrm{low}}$
for a distance at most its length times
$\frac{\sin(\beta+\phi)}{\sin\beta}$,
$\frac{\cos(\beta-\phi)}{\sin\beta}$,
and
$\frac1{\sin\beta}$.

If the two maximum movements
\begin{multline*}
|X \cap (A_1 \cup A_2)|
	\cdot\frac{\sin(\beta+\phi)}{\sin\beta} \cdot \frac{1}{1-w}
+ |X \cap C_0|
	\cdot\frac{\sin(\beta+\phi)}{\sin\beta}
\\
+ |Y \cap (C_0 \cup A_1 \cup A_2 \cup A_3 \cup A_4)|
	\cdot\frac{\cos(\beta-\phi)}{\sin\beta}
+ |Z \cap (C_0 \cup A_1 \cup A_2 \cup A_3 \cup A_4)|
	\cdot\frac1{\sin\beta}
\end{multline*}
and
\begin{multline*}
|X \cap (A_3 \cup A_4)|
	\cdot\frac{\sin(\beta+\phi)}{\sin\beta} \cdot \frac{1}{1-w}
+ |X \cap C_0|
	\cdot\frac{\sin(\beta+\phi)}{\sin\beta}
\\
+ |Y \cap (C_0 \cup A_1 \cup A_2 \cup A_3 \cup A_4)|
	\cdot\frac{\cos(\beta-\phi)}{\sin\beta}
+ |Z \cap (C_0 \cup A_1 \cup A_2 \cup A_3 \cup A_4)|
	\cdot\frac1{\sin\beta}
\end{multline*}
were both less than $1 - 2w$,
then the ADVANCE procedure would find a line that is not blocked.
Without loss of generality, assume that
\begin{equation}\label{eq:xa1234}
|X \cap (A_1 \cup A_2)|
\ge |X \cap (A_3 \cup A_4)|.
\end{equation}
Then we must have
\begin{align} \label{eq:advancex}
|X \cap (A_1 \cup A_2)|
	\cdot\frac{\sin(\beta+\phi)}{\sin\beta} \cdot \frac{1}{1-w}
&+ |X \cap C_0|
	\cdot\frac{\sin(\beta+\phi)}{\sin\beta} \nonumber \\
&+ |Y \cap (C_0 \cup A_1 \cup A_2 \cup A_3 \cup A_4)|
	\cdot\frac{\cos(\beta-\phi)}{\sin\beta} \nonumber \\
&+ |Z \cap (C_0 \cup A_1 \cup A_2 \cup A_3 \cup A_4)|
	\cdot\frac1{\sin\beta} \ge 1 - 2w.
\end{align}

Similarly, assume without loss of generality that
\begin{equation}\label{eq:yb1423}
|Y \cap (B_1 \cup B_4)|
\ge |Y \cap (B_2 \cup B_3)|.
\end{equation}
Then we must have
\begin{align}\label{eq:advancey}
|Y \cap (B_1 \cup B_4)|
	\cdot\frac{\sin(\beta+\phi)}{\sin\beta} \cdot \frac{1}{1-w}
&+ |Y \cap C_0|
	\cdot\frac{\sin(\beta+\phi)}{\sin\beta}  \nonumber \\
&+ |X \cap (C_0 \cup B_1 \cup B_2 \cup B_3 \cup B_4)|
	\cdot\frac{\cos(\beta-\phi)}{\sin\beta} \nonumber \\
&+ |Z \cap (C_0 \cup B_1 \cup B_2 \cup B_3 \cup B_4)|
	\cdot\frac1{\sin\beta} \ge 1 - 2w.
\end{align}

\subsection{The linear program}

We construct a linear program with $32$ linear constraints
corresponding to inequalities~\eqref{eq:xy} through~\eqref{eq:advancey},
and $39$ additional non-negativity constraints for the $39$ variables.
For the LP solver, we wrote a C program that uses \verb|glp_exact|
of the GNU Linear Programming Kit (GLPK 4.52)
compiled with the GNU Multiple Precision Arithmetic Library (GMP 5.1.2).
With parameters $w = 0.1793$ and $\phi = 1.5589^\circ$,
the resulting linear program yields a lower bound of
$2.0000113\ldots > 2 + 10^{-5}$.
The corresponding values of the $39$ variables output by our program
are the following:

\begin{verbatim}
XA1 0.2762651  XA2 0.0726680  XA3 0.1076756  XA4 0.0419541
XB1 0.0000000  XB2 0.0227020  XB3 0.0000000  XB4 0.0000000
XC1 0.1177023  XC2 0.0292085  XC3 0.1469319  XC4 0.0481085  XC0 0.1096004
YA1 0.0000000  YA2 0.0000000  YA3 0.0000000  YA4 0.0911907
YB1 0.1297475  YB2 0.1903349  YB3 0.0000000  YB4 0.2869624
YC1 0.0271035  YC2 0.1387073  YC3 0.0803509  YC4 0.0520305  YC0 0.0000000
ZA1 0.0000000  ZA2 0.0000000  ZA3 0.0000000  ZA4 0.0000000
ZB1 0.0000000  ZB2 0.0000000  ZB3 0.0000000  ZB4 0.0000000
ZC1 0.0000000  ZC2 0.0000000  ZC3 0.0000000  ZC4 0.0000000  ZC0 0.0307674
\end{verbatim}

\section{Conclusion} \label{sec:conclusion}

We have seen that while it is fairly routine to show 
a lower bound of $2$ for the length of an arbitrary barrier for the 
unit square, going beyond this bound poses significant difficulties. 
Here we proved that any segment barrier for the unit square that lies
in a concentric homothetic square of side length $2$ has length 
at least $2+10^{-12}$. In particular, this bound holds for the length
of any interior barrier for the unit square. 

A result of a similar nature from the literature that comes to our
mind is the following.  Let $G$ be an embedded planar graph whose edges are curves.
The \emph{detour} between two points $p$ and $q$ (on edges or vertices)
of~$G$ is the ratio between the length of a shortest path connecting~$p$
and~$q$ in~$G$ and their Euclidean distance~$|pq|$.
The maximum detour over all pairs of points is called the
\emph{geometric dilation} $\delta(G)$; we refer the interested reader
to~\cite{DEG+06,EGK06} for details. 
Ebbers-Baumann, Gr{\"u}ne and Klein~\cite{EGK06} have shown that
every finite point set is contained in a planar graph whose
geometric dilation is at most $1.678$, and some point sets require
graphs with dilation $\delta\ge \pi/2 =1.5707\ldots$.
While obtaining the lower bound  of $\pi/2$ is not extremely difficult, 
it requires nontrivial ideas and it takes a substantial effort to
raise this lower bound to $(1+ 10^{-11})\pi/2$; see~\cite{DEG+06}.  

We conclude with some interesting conjectures and questions on opaques barriers 
that are left open. 
\begin{conjecture} \label{conj:int}
An optimal barrier for the square is interior. 
\end{conjecture}

If Conjecture~\ref{conj:int} were confirmed, Theorem~\ref{T1} would
give a nontrivial lower bound on the length of an arbitrary barrier
for the unit square. At the moment we have such a non-trivial lower bound
only under a suitable locality condition, in particular for interior barriers. 

\begin{itemize} \itemsep 1pt
\item [(1)] Is it possible to adapt the procedure {\sc ADVANCE}, or
  its analysis, in order to deduce a similar lower bound for arbitrary
  (unrestricted) barriers for the unit square?
\end{itemize}

We believe that the leftmost barrier in Figure~\ref{f1} is an optimal 
exterior barrier for the square. 
\begin{conjecture} \label{conj:ext}
The length of an optimal exterior barrier for the unit square is $3$.
\end{conjecture}

It might be interesting to note that the current best barrier for the
disk is exterior, see~\cite{FM86,FMP84}. This suggests three more
questions to include (variant (i) of (2) below is from~\cite{DJP12}).

\begin{itemize} \itemsep 1pt
\item [(2)] Can one give a characterization of the class of convex
polygons whose optimal barriers are 
(i)~interior?
(ii)~exterior?
(iii)~neither interior nor exterior?
\end{itemize}


\begin{thebibliography}{99} 
\itemsep 2pt

\bibitem{A87}
V.~Akman,
An algorithm for determining an opaque minimal forest of a convex polygon,
\emph{Information Processing Letters},
\textbf{24} (1987), 193--198.

\bibitem{AG08} D.~Asimov and J.~L.~Gerver,
Minimum opaque manifolds,
\emph{Geom.\ Dedicata} {\bf 133} (2008), 67--82.

\bibitem{Ba59}
F.~Bagemihl,
Some opaque subsets of a square,
\emph{Michigan Math.\ J.} {\bf 6} (1959), 99--103.

\bibitem{BF85}
I.~B\'ar\'any and Z.~F\"uredi,
Covering all secants of a square,
in \emph{Intuitive Geometry} (G.~Fejes T\'oth, editor),
Colloq. Math. Soc. J\'anos Bolyai, vol. 48 (Si\'ofok, Hungary,1985),
pp.~19--27, North-Holland, Amsterdam, 1987.

\bibitem{Br92}
K.~A.~Brakke,
The opaque cube problem,
\emph{American Mathematical Monthly},
\textbf{99(9)} (1992), 866--871.

\bibitem{Cr69}
H.~T.~Croft,
Curves intersecting certain sets of great-circles on the sphere,
\emph{J. London Math.\ Soc.\ (2)} {\bf 1} (1969), 461--469.

\bibitem{CFG91}
H.~T.~Croft, K.~J.~Falconer, and R.~K.~Guy,
\emph{Unsolved Problems in Geometry},
Springer, New York, 1991.

\bibitem{DO08} E.~D.~Demaine and J.~O'Rourke,
Open problems from CCCG 2007,
in \emph{Proceedings of the 20th Canadian Conference on Computational
Geometry} (CCCG 2008), Montr\'eal, Canada, August 2008, pp. 183--190.

\bibitem{Du88} 
P. Dublish,
An $O(n^3)$ algorithm for finding the minimal opaque forest of a convex polygon, 
\emph{Information Processing Letters}, 
{\bf 29(5)} (1988), 275--276.

\bibitem{DEG+06} 
A. Dumitrescu, A. Ebbers-Baumann, A. Gr\"une, R. Klein, and G. Rote,  
On the geometric dilation of closed curves, graphs, and point sets,
{\em Computational Geometry: Theory and Applications}, 
\textbf{36(1)} (2006), 16--38.

\bibitem{DJP12} 
A. Dumitrescu, M. Jiang, and J. Pach,
Opaque sets, 
{\em Algorithmica}, to appear.
Online first, December 2012; DOI 10.1007/s00453-012-9735-2.

\bibitem{DP10} 
A. Dumitrescu and J. Pach,
Opaque sets, manuscript, May 12, 2010, {\tt arXiv:1005.2218v1}.

\bibitem{EGK06}
A.~Ebbers-Baumann, A.~Gr{\"u}ne, and R.~Klein,
On the geometric dilation of finite point sets,
\emph{Algorithmica},
\textbf{44(2)} (2006), 137--149. 

\bibitem{E82} H.~G. Eggleston, 
The maximal in-radius of the convex cover of a plane connected set of given length, 
\emph{Proc.\ London Math.\ Soc.\ (3)}, \textbf{45} (1982), 456--478. 

\bibitem{EP80} P. Erd\H{o}s and J. Pach,
On  a problem of L. Fejes T\'oth,
\emph{Discrete Mathematics},
\textbf{30(2)} (1980), 103--109.

\bibitem{FM86}
V.~Faber and J.~Mycielski,
The shortest curve that meets all the lines that meet a convex body,
\emph{American Mathematical Monthly},
\textbf{93} (1986), 796--801.

\bibitem{FMP84}
V.~Faber, J.~Mycielski and P. Pedersen,
On the shortest curve which meets all the lines which meet a circle,
\emph{Ann.\ Polon.\ Math.},
\textbf{44} (1984), 249--266.

\bibitem{FT73} L. Fejes T\'oth,
Exploring a planet,
\emph{American Mathematical Monthly},
\textbf{80} (1973), 1043--1044.

\bibitem{FT74} L. Fejes T\'oth,
Remarks on  a dual of Tarski's plank problem,
\emph{Mat.\ Lapok.},
\textbf{25} (1974), 13--20.

\bibitem{F03}
S. R. Finch, \emph{Mathematical Constants},
Cambridge University Press, 2003.

\bibitem{G90} M.~Gardner,
The opaque cube problem, 
\emph{Cubism for Fun} {\bf 23} (March 1990), p. 15.

\bibitem{GM55}
H.~M.~S.~Gupta and N.~C.~B.~Mazumdar,
A note on certain plane sets of points,
\emph{Bull.\ Calcutta Math.\ Soc.} {\bf 47} (1955), 199--201.

\bibitem{H78}
R.~Honsberger,
\emph{Mathematical Morsels},
Dolciani Mathematical Expositions, No.~3,
The Mathematical Association of America, 1978.

\bibitem{J64}
R.~E.~D.~Jones,
Opaque sets of degree $\alpha$,
\emph{American Mathematical Monthly},
\textbf{71} (1964), 535--537.

\bibitem{J80}
H. Joris, Le chasseur perdu dans le foret: une probl\`eme de g\'eom\'etrie plane, 
\emph{Elemente der Mathematik}, \textbf{35} (1980), 1--14.

\bibitem{KMO13}
A. Kawamura, S. Moriyama, and Y. Otachi,
On shortest barriers, communication at the 
\emph{16th Japan Conference on Discrete and Computational Geometry and
Graphs} (JCDCG2 2013),  September 17--19, 2013, Tokyo, Japan. 

\bibitem{K00}
B.~Kawohl,
Some nonconvex shape optimization problems,
in \emph{Optimal Shape Design} (A.~Cellina and A.~Ornelas, editors),
vol. 1740/2000 of Lecture Notes in Mathematics,
Springer, 2000.

\bibitem{KW90}
W.~Kern and A.~Wanka,
On a problem about covering lines by squares,
\emph{Discrete and Computational Geometry},
\textbf{5} (1990), 77--82.

\bibitem{K86} R. Kl\"otzler, 
Universale Rettungskurven I, 
\emph{Zeitschrifte f\"ur Analysis und ihre Anwendungen}, 
\textbf{5} (1986), 27--38.

\bibitem{K87} R. Kl\"otzler and S. Pickenhain, 
Universale Rettungskurven II, 
\emph{Zeitschrifte f\"ur Analysis und ihre Anwendungen}, 
\textbf{6} (1987), 363--369.

\bibitem{KKNX09} E. Kranakis, D. Krizanc, L. Narayanan, K. Xu,
Inapproximability of the perimeter defense problem,
in \emph{Proceedings of the 21st Canadian Conference on Computational
Geometry} (CCCG 2009), Vancouver, Canada, August 2009, pp.~153--156.

\bibitem{M80}
E. Makai, Jr.,
On  a dual of Tarski's plank problem,
\emph{Discrete Geometrie},
\textbf{2}, Kolloq., Inst. Math. Univ. Salzburg, 1980, pp.~127--132.

\bibitem{MP83}
E. Makai, Jr. and J. Pach,
Controlling function classes and covering Euclidean space,
\emph{Studia Scientiarum Mathematicum Hungaricae},
{\bf 18} (1983), 435--459.

\bibitem{Ma16}
S. Mazurkiewicz,
Sur un ensemble ferm\'e, punctiforme, qui rencontre toute droite
passant par un certain domaine (Polish, French summary), \emph{Prace
Mat.-Fiz.} {\bf 27} (1916), 11--16. 


\bibitem{PBTW12}
J. S. Provan, M. Brazil, D. A. Thomas and J. F. Weng,
Minimum opaque covers for polygonal regions,
manuscript, October 2012, {\tt arXiv:1210.8139v1}.
 
\bibitem{RS03}
T.~Richardson and L.~Shepp,
The ``point'' goalie problem,
\emph{Discrete and Computational Geometry},
\textbf{20} (2003), 649--669.

\bibitem{Sa04} 
L. A. Santal\'{o},
\emph{Integral Geometry and Geometric Probability},
2nd edition, Cambridge University Press, 2004.

\bibitem{Sh91} 
T. Shermer,
A counterexample to the algorithms for determining opaque minimal forests, 
\emph{Information Processing Letters}, 
{\bf 40} (1991), 41--42.

\bibitem{Sy1890} 
J. J. Sylvester,
On a funicular solution of Buffon's ``problem of the needle'' in its
most general form, 
\emph{Acta Mathematica}, 
{\bf 14(1)} (1890), 185--205.

\bibitem{V94}
P.~Valtr,
Unit squares intersecting all secants of a square,
\emph{Discrete and Computational Geometry},
\textbf{11} (1994), 235--239.


\end{thebibliography}
\end{document}